\RequirePackage[l2tabu, orthodox]{nag}

\documentclass[12pt]{amsart}
\usepackage{fullpage,url,amssymb,enumerate,colonequals}
\usepackage[all]{xy} % for complicated commutative diagrams
\usepackage{mathrsfs} % for \mathscr (script letters)

\usepackage[OT2,T1]{fontenc}
\DeclareSymbolFont{cyrletters}{OT2}{wncyr}{m}{n}
\DeclareMathSymbol{\Sha}{\mathalpha}{cyrletters}{"58}

\usepackage{color}

\newcommand{\comment}[1]{}
\newcommand{\defi}[1]{\textsf{#1}} % for defined terms

\newcommand{\C}{\mathfrak{C}}
\newcommand{\category}{\mathbf{C}}
\newcommand{\D}{\mathfrak{D}}
\newcommand{\E}{\mathbf{E}}
\newcommand{\Ew}{\E_\omega}
\newcommand{\F}{\mathbb{F}}

\newcommand{\PP}{\mathbb{P}}
\newcommand{\Q}{\mathbb{Q}}
\newcommand{\R}{\mathbb{R}}
\newcommand{\Z}{\mathbb{Z}}
\newcommand{\Qbar}{{\overline{\Q}}}

\newcommand{\calA}{\mathcal{A}}
\newcommand{\calB}{\mathcal{B}}

\newcommand{\calL}{\mathcal{L}}

\newcommand{\FF}{\mathscr{F}}
\newcommand{\FFw}{\FF_\omega}
\newcommand{\GG}{\mathscr{G}}
\newcommand{\GGw}{\GG_\omega}
\newcommand{\II}{\mathscr{I}}

\DeclareMathOperator{\Char}{char}

\DeclareMathOperator{\dom}{dom}

\DeclareMathOperator{\Hom}{Hom}

\DeclareMathOperator{\Isom}{Isom}

\DeclareMathOperator{\Spec}{Spec}

\newcommand{\Fields}{\operatorname{\textup{\textbf{Fields}}}}
\newcommand{\Fieldsw}{\Fields_\omega}
\newcommand{\Graphs}{\operatorname{\textup{\textbf{Graphs}}}}
\newcommand{\Graphsw}{\Graphs_\omega}

\newcommand{\xvec}{\vec{x}}

\newcommand{\la}{\langle}
\newcommand{\ra}{\rangle}
\newcommand{\set}[2]{\ensuremath{ \{ #1 : #2 \} }}

\DeclareMathOperator{\DgSp}{DgSp}
\newcommand{\DS}[2]{\DgSp_{#1}(#2)}

\newcommand{\directsum}{\oplus} % binary direct sum
\newcommand{\injects}{\hookrightarrow}
\newcommand{\isom}{\simeq}

\newcommand{\To}{\longrightarrow}
\newcommand{\union}{\cup} % binary union
\newcommand{\Union}{\bigcup} % union of a collection

\newcommand{\isomto}{\overset{\sim}{\rightarrow}}

\newcommand{\bfd}{\boldsymbol{d}}
\newcommand{\bfz}{\boldsymbol{0}}

\newtheorem{theorem}{Theorem}[section]
\newtheorem{lemma}[theorem]{Lemma}
\newtheorem{corollary}[theorem]{Corollary}
\newtheorem{proposition}[theorem]{Proposition}

\theoremstyle{definition}
\newtheorem{defn}[theorem]{Definition}
\newtheorem{definition}[theorem]{Definition}

\theoremstyle{remark}
\newtheorem{remark}[theorem]{Remark}

\newcommand{\be}{\begin{enumerate}}
\newcommand{\ee}{\end{enumerate}}
\usepackage{microtype}

\usepackage[
	pdfauthor={Russell Miller, Bjorn Poonen, Hans Schoutens, Alexandra Shlapentokh},
	pdftitle={A computable functor from graphs to fields},
]{hyperref}
\usepackage[alphabetic,lite,nobysame]{amsrefs} % for bibliography

\begin{document}

\title[Graphs and fields]{A computable functor from graphs to fields}
\author{Russell Miller}\thanks{This project was initiated at a workshop held at
the American Institute of Mathematics.\\
\indent
The first author was partially supported
by National Science Foundation 
grants DMS-1001306 and DMS-1362206, 
and by PSC-CUNY Research Awards 67839-00 45 and
66582-00 44.} 
\address{Department of Mathematics, Queens College, Queens, NY 11367 \& Ph.D.\ Programs in Mathematics and Computer Science, CUNY Graduate Center, New York, NY 10016, USA}
\email{Russell.Miller@qc.cuny.edu}
\urladdr{http://qcpages.qc.cuny.edu/~rmiller}
\author{Bjorn Poonen}
\thanks{The second author was partially supported by National Science Foundation grant DMS-1069236 and a grant from the Simons Foundation (340694 to Bjorn Poonen).}
\address{Department of Mathematics, Massachusetts Institute of Technology, Cambridge, MA 02139-4307, USA}
\email{poonen@math.mit.edu}
\urladdr{http://math.mit.edu/~poonen}
\author{Hans Schoutens}
\thanks{}
\address{Department of Mathematics, New York City College of Technology, 300 Jay Street, Brooklyn, NY 11201 \& Ph.D.\ Program in Mathematics, CUNY Graduate Center, New York, NY 10016, USA}
\email{hschoutens@citytech.cuny.edu}
\urladdr{}
\author{Alexandra Shlapentokh}
\thanks{The fourth author was partially supported by National Science Foundation grant DMS-1161456. \\
\indent  Any opinions, findings, and conclusions or recommendations expressed in this material are those of the authors and do not necessarily reflect the views of the National Science Foundation or Simons Foundation.}
\address{Department of Mathematics, East Carolina University, Greenville, NC 27858, USA}
\email{shlapentokha@ecu.edu}
\urladdr{http://myweb.ecu.edu/shlapentokha/}
%\date{\today}
\date{October 25, 2015}

\begin{abstract}
We construct a fully faithful functor from the category of graphs 
to the category of fields.
Using this functor,
we resolve a longstanding open problem in computable model theory, 
by showing that for every nontrivial countable structure $\mathcal{S}$,
there exists a countable field $\mathcal{F}$ with the same essential
computable-model-theoretic properties as $\mathcal{S}$.
Along the way, we develop a new ``computable category theory'',
and prove that our functor and its partially-defined inverse
(restricted to the categories of countable graphs and countable fields) 
are computable functors.
\end{abstract}

\maketitle

%****************************************************************************
\section{Introduction}\label{S:introduction}

\subsection{A functor from graphs to fields}

Let $\Graphs$ be the category of symmetric irreflexive graphs
in which morphisms are isomorphisms onto induced subgraphs
(see Section~\ref{S:notation graphs}).
Let $\Fields$ be the category of fields,
with field homomorphisms as the morphisms.
Using arithmetic geometry, we will prove the following:

\begin{theorem}
\label{T:fully faithful}
There exists a fully faithful functor $\FF \colon \Graphs \to \Fields$.
\end{theorem}

\noindent
(The definitions of ``full'' and ``faithful'' are reviewed in
Section~\ref{S:category theory}.)
In particular, given a graph $G$, the functor produces a field
with the same automorphism group as $G$.

\subsection{Computable functors}
\label{S:computable functors}

For applications to computable model theory,
we are interested in graphs and fields whose 
underlying set is $\omega \colonequals \{0,1,2,\ldots\}$.
These form full subcategories $\Graphsw$ and $\Fieldsw$.
The functor $\FF$ of Theorem~\ref{T:fully faithful} will be constructed 
so that it restricts to a functor $\FFw \colon \Graphsw \to \Fieldsw$.
Let $\Ew$ denote the essential image of $\FFw$ 
(see Section~\ref{S:category theory} for definitions).
Then we may view $\FFw$ as a functor from $\Graphsw$ to $\Ew$.
We will also a define a functor $\GGw \colon \Ew \to \Graphsw$,
and we would like to say that $\FFw$ and $\GGw$ are computable
and are inverse to each other in some computable way.

To guide us to the correct formulation of such statements,
we create a new ``type-2 computable category theory'';
see Section~\ref{sec:categories}.
The adjective ``type-2'', borrowed from computable analysis 
(see Remark~\ref{R:type-2}),
indicates that we work with noncomputable objects;
indeed, $\Graphsw$ and $\Fieldsw$ each contain uncountably many 
noncomputable objects.
The effectiveness in our definitions 
really arises in the concept of a \emph{computable functor}, 
a functor in which the processes of transforming objects to objects
and morphisms to morphisms are given by Turing functionals:
roughly speaking, 
each output should be computable given an oracle for the input
(see Definition~\ref{defn:computablefunctor}).
Our computable category theory includes also a notion of
\emph{computable isomorphism of functors}
(see Definition~\ref{D:computable morphism}).
Using this lexicon, we can now state our main result 
on the computability of the functors:

\begin{theorem}
\label{T:computable functors}
\hfill
\begin{enumerate}[\upshape (a)]
\item \label{I:Fw is computable}
The functors $\FFw$ and $\GGw$ are computable 
in the sense of Definition~\ref{defn:computablefunctor}.
\item \label{I:Fw and Gw are inverse}
The composition $\GGw \FFw$ is computably isomorphic to $1_{\Graphsw}$, 
and $\FFw \GGw$ is computably isomorphic to $1_{\Ew}$.
\end{enumerate}
\end{theorem}

We may summarize Theorem~\ref{T:computable functors}
by saying that $\FFw$ and $\GGw$ give a computable equivalence
of categories between $\Graphsw$ and $\Ew$
(see Definition~\ref{D:equivalence}).
In fact, we will define $\GGw$ so that $\GGw \FFw$ \emph{equals} $1_{\Graphsw}$, 
but computable isomorphism in place of equality here 
suffices for the applications in Section~\ref{sec:consequences}.

Here is one concrete consequence of Theorem~\ref{T:computable functors}:

\begin{corollary}
\label{C:G and f are computable}
If a field $F \in \Fieldsw$ is isomorphic to a field 
in the image of $\FFw$, 
then one can compute from $F$ a graph $G \in \Graphsw$
and an isomorphism $F \to \FFw(G)$.
\end{corollary}

\begin{proof}
Apply the isomorphism $\GGw \FFw \isom 1_{\Graphsw}$ 
of Theorem~\ref{T:computable functors}\eqref{I:Fw and Gw are inverse}
to $G \colonequals \GGw(F)$.
\end{proof}

Specializing Corollary~\ref{C:G and f are computable} 
to the case in which $F$ is computable yields the following:

\begin{corollary}
\label{cor:extrafields}
Every computable field isomorphic to a field in the image of $\FFw$
is \emph{computably} isomorphic to $\FFw(G)$ 
for some \emph{computable} $G \in \Graphsw$.
\end{corollary}

Following \cite{K86}*{\S 4}, we call a structure \defi{automorphically trivial}
if there exists a finite subset $S_0$ of its domain $S$
such that every permutation of $S$ fixing $S_0$ pointwise
is an automorphism.

\begin{proposition}
\label{prop:extrafields}
Let $F \in \Ew$;
that is, $F$ is isomorphic to a field in the image of~$\FFw$.
Let $\II \colonequals \{G \in \Graphs_w : \FFw(G) \isom F\}$,
so $\II \ne \emptyset$.
\begin{enumerate}[\upshape (a)]
\item \label{I:I is isomorphism class}
The set $\II$ is an isomorphism class in $\Graphs_w$.
\item \label{I:dichotomy}
Either every graph in $\II$ is automorphically trivial and computable,
or there exists a graph in $\II$ of the same Turing degree as $F$.
\end{enumerate}
\end{proposition}

Proofs of Theorem~\ref{T:computable functors}
and Proposition~\ref{prop:extrafields} appear in Section~\ref{S:computability}.

\begin{remark}
As will be explained in Section~\ref{S:related work},
there have been other attempts to give effective versions of
category theory (and, more succesfully, to give a categorical underpinning
to computability theory).
To our knowledge, however, ours is the first attempt
to define effectiveness for functors using type-2 Turing computation.
\end{remark}

\subsection{Computable model theory}
\label{S:computable model theory}

One of the goals of computable model theory is to help
distinguish various classes of countable structures
according to the algorithmic complexity of those structures.
The class of algebraically closed fields of characteristic $0$, for example,
is viewed throughout model theory as a particularly simple class,
and its computability-theoretic properties confirm this view:
every countable algebraically closed field can be computably presented,
all of them are relatively $\Delta^0_2$-categorical, and the only one
that is not relatively computably categorical has infinite computable dimension.
(All these terms are defined in Section~\ref{sec:consequences}.)
In contrast, the theory of linear orders is a good deal more complex:
there do exist countable linear orders with no computable presentation,
and linear orders with much higher degrees of categoricity than $\bfz'$.
In this view, the theory of graphs is even more complicated:  for example,
every computable linear order has computable dimension either $1$
or $\omega$, whereas for computable (symmetric irreflexive) graphs
all computable dimensions $\leq\omega$ are known to occur.

We will discuss specific properties 
such as computable presentability and computable dimension
when we come to prove results
about them, in Section~\ref{sec:consequences}.  For now,
we simply note that a substantial body of results has been established
on the possibility of transferring these properties from one class
of countable structures to another.  After much piecemeal work
by assorted authors, most of these results were gathered
together and brought to completion in the work \cite{HKSS},
by Hirschfeldt, Khoussainov, Shore, and Slinko.  There it was proven that
the class of symmetric irreflexive graphs is \emph{complete},
in the very strong sense of their Definition~1.21.
The authors gave a coding
procedure that, given any countable structure $S$ with domain $\omega$
(in an arbitrary computable language, with one quite trivial restriction) 
as its input,
produced a countable graph $G$ on the same domain with the same
computable-model-theoretic properties as $S$.
Several other natural properties have been introduced since then
(the automorphism spectrum, in \cite{HMM}*{Definition~1.1},
and the categoricity spectrum, in \cite{FKM10}*{Definition~1.2},
for example), and each of these has also turned out to be preserved
under the construction from \cite{HKSS}.
The method they gave was quite robust, in this sense, and one
may expect that it will also be found to preserve other properties
that are yet to be defined.

Having established the completeness of the class of countable graphs in this sense,
the authors went on to consider many other everyday classes of countable
first-order structures.  By doing a similar coding from graphs into other classes,
they succeeded in proving the completeness (in this same sense) of the following classes:
\begin{itemize}
\item
countable directed graphs;
\item
countable partial orderings;
\item
countable lattices;
\item
countable rings (with zero-divisors);
\item
countable integral domains of arbitrary characteristic;
\item
countable commutative semigroups;
\item
countable $2$-step nilpotent groups.
\end{itemize}
(We note again that in some cases, these results had been established
in earlier work by other authors.  In certain of these cases, the language
must be augmented by a finite number of constant symbols.)
On the other hand, various existing and subsequent results demonstrated
that none of the following classes is complete in this sense:
\begin{itemize}
\item
countable linear orders 
(e.g., by results in \cite{GD80}*{Theorem~2}, 
\cite{Rem81}*{Corollary~1}, and \cite{Ri81}*{Theorem~3.3});
\item
countable Boolean algebras (see \cite{DJ94}*{Theorem~1}, 
or use \cite{Ri81}*{Theorem~3.1});
\item
countable trees, either as partial orders or under the meet relation $\wedge$
(by \cite{Ri81}*{Theorem~3.4} or \cite{LMMS}*{Theorem~1.8});
\item
algebraic field extensions of $\Q$ or of $\F_p$
(see \cite{M09}*{Corollary~5.5}, for example);
\item
field extensions of $\Q$ of finite transcendence degree
(see \cite{M09}*{\S 6});
\item
countable archimedean ordered fields 
(by Theorem~3.3.1 in Levin's thesis \cite{Levin}).
\end{itemize}
In addition, 
Goncharov announces in \cite{G81} that every computable abelian group
has computable dimension $1$ or $\omega$, though it seems that 
a detailed proof does not exist in print.  This result would add
the class of countable abelian groups to the list of non-complete classes.

Conspicuously absent from both of the lists above is the class of countable fields.  
The question whether this class possesses the completeness property 
has remained open, 
despite substantial interest since the publication of \cite{HKSS}: 
see the introduction to \cite{LMMS}.
In this paper we resolve the question, using the computable functors
of Section~\ref{S:computable functors}:

\begin{theorem}
\label{T:fields are complete}
The class of countable fields has the completeness property 
of \cite{HKSS}*{Definition~1.21}.
\end{theorem}

Proving the many specific aspects of this theorem 
will take most of Section~\ref{sec:consequences}.
Our proof of Theorem~\ref{T:fields are complete} shows more specifically 
that the class of countable fields of characteristic~$0$ 
has the completeness property.

% \bjorn{Comment about positive characteristic to be added here when [MPSS2] is publicly available:}
% The paper \cite{MPSS2} proves an analogous result for the class of fields of characteristic~$p>0$, using a technique that differs in some ways from the techniques of this paper.

\begin{remark}
The constructions in \cite{HKSS} too
can be expressed in terms of our computable category theory.
\end{remark}

\subsection{Structure of the paper}

Section~\ref{S:notation} introduces
notation and definitions.
Section~\ref{sec:categories} introduces the key definitions
of our computable category theory,
and discusses its relation to other work in the literature.

Section~\ref{S:curves} defines some algebraic curves 
used in Section~\ref{S:construction} to construct $\FF$,
and proves arithmetic properties of these curves
that are used in Section~\ref{S:inverse functor} to construct $\GG$
and prove enough properties of $\FF$ and $\GG$
to prove Theorem~\ref{T:fully faithful}.
Section~\ref{S:computability} constructs 
the computable analogues $\FFw$ and $\GGw$,
and proves Theorem~\ref{T:computable functors}.
Section~\ref{S:ordered fields} is something of a side remark:
it proves that for every $G \in \Graphsw$,
the field $\FFw(G)$ is isomorphic to a subfield of $\R$;
but $\FFw$ cannot be viewed as a functor to the category of ordered fields
if morphisms in the latter are required to respect the orderings.
Finally, Section~\ref{sec:consequences} 
explains the implications of our functors for computable model theory.
Many of the results there follow formally 
from \cite{HKSS}*{Theorem~3.1} and Theorem~\ref{T:computable functors},
so they could be stated more generally in terms of any computable
equivalence of categories of structures, but for concreteness, we state them
specifically for the categories of graphs and fields.

%****************************************************************************
\section{Notation and definitions}\label{S:notation}

\subsection{Graphs}
\label{S:notation graphs}

Given a set $I$, let $\binom{I}{2}$ be the set of $2$-element subsets of $I$.
A (symmetric, irreflexive) \defi{graph} 
is a set $V$ equipped with a subset $E \subseteq \binom{V}{2}$;
then $\{(i,j) \in V \times V: \{i,j\} \in E\}$
is a symmetric irreflexive relation on $V$.
If $G$ is a graph $(V,E)$, then $\#G \colonequals \#V$;
call $G$ finite if $V$ is finite.
Define a \defi{morphism of graphs} $(V,E) \to (V',E')$ to be an
injection $f \colon V \injects V'$ 
such that for each $\{i,j\} \in \binom{V}{2}$,
we have $\{i,j\} \in E$ if and only if $\{f(i),f(j)\} \in E'$.
In other words, a morphism of graphs $(V,E) \to (V',E')$ is 
an isomorphism from $(V,E)$ onto an induced subgraph of $(V',E')$.
These notions define a category $\Graphs$.

\subsection{Category theory}
\label{S:category theory}

A \defi{full subcategory} of a category $\category$
is a category consisting of some the objects of $\category$
but all of the morphisms between pairs of these chosen objects.
Let $\FF \colon \category \to \mathbf{D}$ be a functor.
Then $\FF$ is \defi{full} (respectively, \defi{faithful})
if for any two objects $C_1,C_2 \in \category$,
the map $\Hom(C_1,C_2) \to \Hom(\FF(C_1),\FF(C_2))$
is surjective (respectively, injective).
The \defi{essential image} of $\FF$ is the full subcategory of $\mathbf{D}$
consisting of objects $D \in \mathbf{D}$ that are isomorphic
to $\FF(C)$ for some $C \in \category$.

\subsection{Fields and arithmetic geometry}

Let $\Fields$ be the category of fields,
with field homomorphisms as the morphisms.
If $X$ is a $k$-variety, and $L \supseteq k$ is a field extension,
let $X_L$ denote the base extension;
concretely, $X_L$ is the $L$-variety defined by the same equations as $X$,
but with coefficients considered to be in $L$.
A $k$-variety is \defi{integral} if it is irreducible
and the ring of functions on any Zariski open subset has no nilpotent elements.
A $k$-variety is \defi{geometrically integral} if it remains integral
after any extension of the ground field;
for example, the affine curve $x^2-2y^2=0$ over $\Q$ is integral
but not geometrically integral.
If $X$ is an integral $k$-variety, let $k(X)$ denote its function field.
If $(X_i)_{i \in I}$ is a collection of geometrically integral $k$-varieties
indexed by a set $I$,
let $k(\prod X_i)$ denote the direct limit of the function fields
$k(\prod_{i \in J} X_i)$ over all finite subsets $J \subseteq I$.
If $X$ is a geometrically integral curve, 
let $g_X \in \Z_{\ge 0}$ be its \defi{geometric genus},
defined as the dimension of the space of global $1$-forms
on the regular projective model of~$X$.

\subsection{Computable model theory}
\label{S:notation computable model theory}

Given a structure $M$, let $\dom(M)$ be its \defi{domain} (underlying set),
and let $\Delta(M)$ be its \defi{atomic diagram}
(the set of atomic sentences true in $M$).
If $M$ is a countable structure with $\dom(M)=\omega$, 
let $\deg M$ be the Turing degree of $\Delta(M)$,
and define the \defi{spectrum} of $M$ as
\label{defn:spectrum} 
$\Spec M \colonequals \{\deg N: N \isom M \textup{ and } \dom(N) = \omega\}$.
Given two such structures $M$ and $N$,
we say that $M$ is \defi{computable from $N$}, or write $M \le_T N$,
if $\Delta(M)$ is computable under a $\Delta(N)$-oracle,
or equivalently $\deg M \le_T \deg N$.

%****************************************************************************

\section{Computable category theory}
\label{sec:categories}

\subsection{Computable functors}

By a \defi{category of structures on $\omega$},
we mean a subcategory $\category$ 
of the category of all first-order $\calL$-structures
with domain $\omega \colonequals \{0,1,2,\ldots\}$,
for some computable language $\calL$.
(In some cases, one might allow also finite subsets of $\omega$ as domains.)
Here are two examples:
\begin{itemize}
\item 
Let $\Graphsw$ be the category whose objects are 
the (symmetric irreflexive) graphs having underlying set $\omega$,
and whose morphisms are isomorphisms onto induced subgraphs.
\item
Let $\Fieldsw$ be the category whose objects are
the fields having underlying set $\omega$,
and whose morphisms are field homomorphisms.
\end{itemize}

There is nothing particularly effective about these categories.  
The requirement that the domain equal $\omega$
gives us the opportunity to consider computability questions about the
structures in a category, but a graph $G$ on $\omega$ with a noncomputable
edge relation would be an object of $\Graphsw$ in perfectly good standing.
Also, even for computable sets $E \subseteq \omega^2$, there is no general
way to determine whether $E$ actually is the edge relation of a symmetric irreflexive graph.

Effectiveness considerations arise when we consider functors
between these categories.  A functor maps objects to objects and morphisms
to morphisms, and we would like this process to be effective.
\begin{defn}
\label{defn:computablefunctor}
Let $\category$ and $\category'$ be categories of structures on $\omega$
(with respect to possibly different languages $\calL$ and $\calL'$).
A \defi{computable functor} is a functor $\FF \colon \category \to \category'$ 
for which there exist Turing functionals $\Phi$ and $\Psi$ such that
\begin{enumerate}[\upshape (i)]
\item\label{I:condition 1}
for every $S\in\category$, the function $\Phi^S$ computes 
(the atomic diagram of) the structure $\FF(S)$; and
\item\label{I:condition 2}
for every morphism $g \colon S\to T$ in $\category$, 
we have $\Psi^{S\oplus g\oplus T}=\FF(g)$ in $\category'$.
\end{enumerate}
\end{defn}

Computing the atomic diagram of $\FF(S)$ is equivalent
to computing every function and relation from $\calL'$ on $\FF(S)$ (and, in case $\calL'$
is infinite, doing so uniformly and also computing the value of each constant symbol).
One could state \eqref{I:condition 1} by saying that, 
for every $n\in\omega$, the function
$\Phi^S(\la n,\xvec\ra)$ is the value on input $\xvec$ 
of the $n$th symbol of the language $\calL'$, as interpreted in $\FF(S)$.  For relation symbols, this value is Boolean, while
for constants and function symbols, it is an element of $\FF(S)$, i.e., a natural number.
Condition~\eqref{I:condition 2} 
states that $\Psi^{S\oplus g\oplus T}$ computes
the morphism $\FF(g)$ in $\category'$, 
viewed as a map from $\dom(S)$ into $\dom(T)$.
Note that one is allowed to use the objects $S$ and $T$, 
in addition to the morphism $g$ between them, to compute $\FF(g)$;
this is natural if one thinks of the source and target 
as being part of the data describing a function.
Our proofs in Section~\ref{S:computability} 
that $\FFw$ and $\GGw$ are computable
also illustrate why including $S$ and $T$ is appropriate.

\begin{remark}
\label{R:type-2}
Our computable functors may also be called \defi{type-2 computable functors}.
The terminology ``type-2'' comes from computable analysis, where this
concept of computability is common; 
the book \cite{PER} is a standard reference for this topic.  
Computable analysis considers real numbers,
given by fast-converging Cauchy sequences of rationals.
Two such sequences may well represent the same real number, 
and equality of real numbers is not computable;
that is, no Turing functional $\Phi$
has the property that, for all fast-converging Cauchy sequences $f$ and $g$
with respective limits $x$ and $y$ in $\mathbb R$,
\[
	 \Phi^{f\oplus g}(0) = 
	\begin{cases}
		 1, &\text{ if $x=y$;}\\
		 0, &\text{ if $x \ne y$;}
	\end{cases}
\]
cf.~\cite{PER}*{p.~23, Fact~3}.
On the other hand, there does exist a Turing functional $\Psi$ such that,
for all such $f$ and $g$, 
$\Psi^{f\oplus g}$ is itself a fast-converging Cauchy sequence
with limit $x+y$, so addition is a type-2 computable function on $\R$,
as are the other arithmetic operations 
and many more standard functions on $\R$.
Likewise, here in category theory, 
we are computing functions not between natural numbers,
but between subsets of natural numbers.  This is clear in the case of the edge
relation on a graph, but in fact all countable structures 
(in countable first-order languages) 
turn out to be represented by subsets of $\omega$, 
just as all real numbers are.
\end{remark}

\subsection{Computable morphisms of functors}

\begin{definition}
\label{D:computable morphism}
Let $\FF_1, \FF_2 \colon \category \to \category'$ 
be computable functors.
A \defi{computable morphism of functors} 
(or \defi{computable natural transformation}) from $\FF_1$ to $\FF_2$
is a morphism of functors $\tau \colon \FF_1 \to \FF_2$
such that there exists a Turing functional that on input $S \in \category$
computes the morphism $\tau(S) \colon \FF_1(S) \to \FF_2(S)$.
\end{definition}

This notion leads in the usual way to the notion of 
\defi{computable isomorphism of functors}: this is 
a computable morphism of functors 
having a two-sided inverse that is again a computable morphism of functors.
Then, adapting the definition of equivalence of categories leads to
the following.

\begin{definition}
\label{D:equivalence}
Let $\category$ and $\category'$ be categories of structures on $\omega$.
A \defi{computable equivalence of categories} from $\category$ to $\category'$
is a pair of computable functors $\FF \colon \category \to \category'$
and $\GG \colon \category' \to \category$
together with computable isomorphisms $\GG\FF \to 1_{\category}$
and $\FF\GG \to 1_{\category'}$.
\end{definition}

In Definition~\ref{D:equivalence},
we may refer to $\FF$ alone as a computable equivalence of categories
if the other functors and isomorphisms exist.

\subsection{Related work}
\label{S:related work}

To a certain extent, the concept of computable functor 
has appeared before in computable model theory.
\defi{Turing-computable reducibility} 
(as defined, e.g., in \cite{KMV}*{\S 1}, by Knight, Miller, and Vanden Boom)
can be viewed as a version of it.
In that definition, one class $\C$ of computable structures
is Turing-computably reducible (or TC-reducible) to another such class $\D$
if there exists a Turing functional $\Phi$ that accepts as input the atomic
diagram $\Delta(S)$ of a structure $S\in\C$ and outputs the atomic diagram
$\Phi^{\Delta(S)}=\Delta(T)$ of a structure $T\in\D$.  Writing $\FF(S)$
for $T$, the further requirement is that $S\cong S'$ if and only if $\FF(S)\cong\FF(S')$.
So this definition essentially includes the first half of Definition
\ref{defn:computablefunctor} above, although stated only for
computable structures, not for arbitrary structures with domain $\omega$.
The second half is related to the preservation of the isomorphism relation,
but here Definition~\ref{defn:computablefunctor} is far stronger, requiring
the actual computation of an isomorphism $\FF(g)$ from $\FF(S)$ onto $\FF(S')$,
given an isomorphism $g$ from $S$ onto $S'$.  It would be reasonable to
investigate Definition~\ref{defn:computablefunctor} more fully,
especially in light of the work in~\cite{KMV}.

In parallel with our research into functors from graphs to fields,
Montalb\'an examined a notion that he calls
\defi{effective interpretation}.  This is detailed in \cite{M14}*{\S 5},
and is very much in the vein of the traditional model-theoretic
notion of interpretation, with effectiveness conditions added.
The functor we build in Section~\ref{S:construction} will allow
an effective interpretation of each graph $G$ in the field $\FF(G)$,
since the domain of each graph will be definable uniformly
by an existential formula in the corresponding field and the
edge relation of the graph will likewise be uniformly existentially
definable on that domain within that field.  Indeed, $\FF$ will have a
computable left-inverse functor, and it will actually be the case
that $\FF(G)$ always has an effective interpretation in $G$,
although this is not so obvious at first glance.  
These are examples of a more general result, 
obtained in \cite{HTM3}*{Theorems 5 and~12}
by Harrison-Trainor, Melnikov, Miller, and Montalb\'an 
in the wake of the present work, 
that the existence of a computable functor from
$\category$ to $\category'$ 
is equivalent to the uniform effective interpretability
of all elements of the class $\category'$ in the class $\category$.
(So the effective interpretation of $G$ in $\FF(G)$ is an outgrowth
of the computable left inverse functor for $\FF$, not of $\FF$ itself.)

The results in \cite{HKSS}
are proven largely by the construction of computable functors, although
not described in that way.  However, one could also ask the same questions
about categories known not to be complete.  For example, there is a
natural construction of a Boolean algebra $\FF(L)$ from a linear order $L$,
simply by taking the interval algebra of $L$, where the morphisms in each category
are simply homomorphisms of the structures.  On its face, this functor appears
to be neither full nor faithful, based on known results, and it does not have
a precise computable inverse functor on its image, although it may
come close to doing so.  It cannot have all of these properties,
because there does exist a linear order whose spectrum is not realized
by any Boolean algebra, as shown by Jockusch and Soare 
in \cite{JS}*{Theorem~1}.  (Here we use 
a generalization of the result in this article, namely,
that a computable equivalence of categories 
onto a strictly full subcategory
allows one to transfer spectra
from objects of the first category to objects of the second.
This generalization appears to have a straightforward direct proof,
and in any case it follows from effective bi-interpretability, hence
from \cite{HTM3}*{Theorem~12}, using \cite{M14}*{Lemma~5.3}.)
We suspect that similar results distinguishing
the properties of various everyday classes of countable structures
may yield further insights into effectiveness, fullness, faithfulness, and other
properties of functors among these classes, especially the incomplete ones listed in 
Section~\ref{S:computable model theory}.

%****************************************************************************
\section{Curves}\label{S:curves}

Define polynomials
\begin{align*}
	p(u,v) &\colonequals u^4+16uv^3+10v^4+16v-4  \in \Q[u,v], \textup{ and}\\
	q(T,x,y) &\colonequals x^4+y^4+1 + T(x^4+xy^3+y+1) \in \Q[T,x,y].
\end{align*}
Let $X$ be the affine plane curve $p(u,v)=0$ over $\Q$.
For any field $F$ and $t \in F$,
let $Y_t$ be the affine plane curve $q(t,x,y)=0$ over $F$.
If we take $F=\Q(T)$ (the field of rational functions in one indeterminate)
and $t=T$, then we obtain a curve $Y_T$ over $\Q(T)$;
let $Y = Y_T$.
More generally, if $F \supseteq \Q$ and $t$ is transcendental over $\Q$,
then $Y_t$ is the base change of $Y$
by the field homomorphism $\Q(T) \to F$ sending $T$ to~$t$,
so $Y_t$ inherits many properties from $Y$.

The properties we need of these curves to construct the functor $\FF$
are contained in 
parts \eqref{I:geometrically integral}--\eqref{I:Y has no real points}
of the following lemma.
Later, to prove that for every $G \in \Graphsw$ 
the field $\FF(G)$ is isomorphic to a subfield of $\R$,
we will also need~\eqref{I:20 to infinity}.

\begin{lemma}
\label{L:properties of the curves}
\hfill
\begin{enumerate}[\upshape (1)]
\item \label{I:geometrically integral} 
Both $X$ and $Y$ are geometrically integral.
\item \label{I:same genus} 
We have $g_X = g_Y > 1$.
\item \label{I:no automorphisms} 
Even after base field extension,
$X$ and $Y$ have no nontrivial birational automorphisms.
\item \label{I:non-isotrivial}
Even after base field extension,
there is no birational map from $Y$ to any curve definable over a 
finite extension of $\Q$.
\item \label{I:X(Q)}
We have $X(\Q)=\emptyset$.
\item \label{I:X has real points}
We have $u(X(\R))=\R$.
\item \label{I:Y has no real points}
There exists an open neighborhood $U$ of $0$ in $\R$
such that for each $t \in U$
we have $Y_t(\R) = \emptyset$.
\item \label{I:20 to infinity}
If $t \in [20,\infty)$, 
then $Y_t$ has a real point with $x$-coordinate in $[1,2]$.
\end{enumerate}
\end{lemma}

\begin{proof}
The projective closure $\widetilde{X}$ of $X$ 
specializes under reduction modulo $5$ to the curve
\[
	\widetilde{X}_5 \colon u^4 + uv^3 + vw^3 + w^4 = 0
\]
in $\PP^2_{\F_5}$.
The projective closure of $Y$ specializes at $T=0$ and $T=\infty$
to the curves 
\[
	\widetilde{Y}_0 \colon x^4+y^4+z^4=0 \qquad \textup{and} \qquad
	\widetilde{Y}_\infty \colon x^4 + xy^3 + yz^3 + z^4=0
\]
in $\PP^2_\Q$, respectively.
By \cite{Poonen-without3}*{Case~I with $n=2$, $d=4$, $c=1$}, 
$\widetilde{X}_5$ and $\widetilde{Y}_\infty$
are smooth, projective, geometrically integral plane curves of genus~$3$
with no nontrivial birational automorphisms even after base extension.
It follows that $X$ and $Y$ have the same properties,
except for being projective.
\begin{enumerate}[\upshape (1)]
\item Explained above.
\item By the above, $g_X=g_Y=3$.
\item Explained above.
\item If there were such a birational map,
then the specializations $\widetilde{Y}_0$ and $\widetilde{Y}_\infty$
would be isomorphic over $\Qbar$.
But the former has a nontrivial automorphism $(x,y,z) \mapsto (-x,y,z)$.
\item 
The given model of $\widetilde{X}$ (viewed over $\Z$) reduces
modulo~$8$ to $u^4 + 2 v^4 + 4 w^4 = 0$,
which has no solutions in $\PP^2(\Z/8\Z)$.
Now $X(\Q) \injects \widetilde{X}(\Q) = \widetilde{X}(\Z) \to \widetilde{X}(\Z/8\Z) = \emptyset$.
\item 
Fix $u \in \R$.
Then $\lim_{v \to +\infty} p(u,v) = +\infty$, 
so by the intermediate value theorem it suffices to find $v \in \R$
such that $p(u,v)<0$.
If $|u| < \sqrt{2}$, then $v=0$ works.
If $|u| \ge \sqrt{2}$, then $p(u,-u) = -5u^4 - 16u - 4 < 0$ by calculus.
\item 
This follows from $\widetilde{Y}_0(\R)=\emptyset$ and compactness.
\item 
Suppose that $t \ge 20$ and $x \in [1,2]$.
Then $q(t,x,-3) \le 2^4+3^4+1 + 20(2^4-27-3+1) < 0$ but $q(t,x,0) > 0$,
so there exists $y \in \R$ with $q(t,x,y)=0$.\qedhere
\end{enumerate}
\end{proof}

%****************************************************************************
\section{Construction of the functor}\label{S:construction}

\subsection{Construction}
\label{S:construction of F}

We now define a functor $\FF \colon \Graphs \to \Fields$.
Suppose that we are given a graph $G = (V,E)$.
Let $K = \Q(\prod_{i \in V} X)$.
Let $u_i,v_i \in K$ correspond to the rational functions $u,v$ 
on the $i$th copy of $X$.
Thus $(u_i)_{i\in V}$ is a transcendence basis for $K/\Q$.
For $\{i,j\} \in \binom{V}{2}$,
define the $K$-curve $Z_{\{i,j\}}$ as $Y_{u_i u_j}$ if $\{i,j\} \in E$,
and $Y_{u_i+u_j}$ if $\{i,j\} \notin E$.
Define $\FF(G) \colonequals K(\prod Z_{\{i,j\}})$,
where the product is over $\{i,j\} \in \binom{V}{2}$.
A morphism $G \to G'$ induces 
an obvious field homomorphism $\FF(G) \to \FF(G')$.
We obtain a functor $\FF$.

\begin{remark}
If $G$ is finite, then $\#\FF(G)=\aleph_0$.
If $G$ is infinite, then $\#\FF(G)=\#G$.
\end{remark}

\subsection{Properties}

Here we prove properties of $\FF(G)$
that will enable us to recover $G$ from $\FF(G)$,
or more precisely, to prove that $\FF$ is fully faithful.
In the proofs in this section, 
labels like \eqref{I:same genus}
refer to the parts of Lemma~\ref{L:properties of the curves}.

\begin{lemma}
\label{L:no rational maps}
Fix $G$.
Let $L$ be any field extension of $K$.
Consider the base changes to $L$ of 
$X$ and all the curves $Y_{u_i u_j}$ and $Y_{u_i+u_j}$.
The only nonconstant rational maps between these curves over $L$
are the identity maps from one of them to itself.
\end{lemma}

\begin{proof}
By \eqref{I:same genus}, all the curves have the same genus.
By \eqref{I:no automorphisms} and Lemma~\ref{L:two curves}, 
it suffices to show that no two distinct curves in this list
are birational even after base field extension.
By~\eqref{I:non-isotrivial}, this is already true for $X$ and $Y_t$
for any transcendental $t$.
If $t,t'$ are algebraically independent over $\Q$,
and $Y_t$ and $Y_{t'}$ become birational after base field extension,
then we can specialize $t'$ to an element of $\Qbar$
while leaving $t$ transcendental, contradicting~\eqref{I:non-isotrivial}.
The previous two sentences apply in particular to any $t$ and $t'$
taken from the $u_i u_j$ and the $u_i+u_j$.
\end{proof}

\begin{lemma}
\label{L:points over F(G)}
Let $G=(V,E)$ be a graph.
Let $x_{ij},y_{ij} \in \FF(G)$ correspond to the 
rational functions $x,y$ on $Z_{\{i,j\}}$.
\begin{enumerate}[\upshape (i)]
\item \label{I:points of X}
We have $X(\FF(G)) = \{(u_i,v_i) : i \in V\}$.
\item \label{I:points of Y for edges in E}
If $\{i,j\} \in E$,
then $Y_{u_i u_j}(\FF(G)) = \{(x_{ij},y_{ij})\}$
and $Y_{u_i+u_j}(\FF(G))=\emptyset$.
\item \label{I:points of Y for edges outside E}
If $\{i,j\} \notin E$,
then $Y_{u_i u_j}(\FF(G)) = \emptyset$ 
and $Y_{u_i+u_j}(\FF(G))= \{(x_{ij},y_{ij})\}$.
\end{enumerate}
\end{lemma}

\begin{proof}
\hfill
\begin{enumerate}[\upshape (i)]
\item 
By definition, $\FF(G)$ is the direct limit of $K(Z)$,
where $Z$ ranges over finite products of the $Z_{\{i,j\}}$.
Thus, by Lemma~\ref{L:points over function field}, 
each point in $X(\FF(G))$ corresponds to a rational map from some $Z$ 
to the base change $X_K$.
By Lemmas \ref{L:rational maps from a product}
and~\ref{L:no rational maps}
every such rational map is constant.
In other words, $X(\FF(G))=X(K)$.

Similarly, 
by Lemma~\ref{L:points over function field}, 
each point in $X(K)$ corresponds to a rational map
from some finite power of $X$ to $X$.
By~\eqref{I:X(Q)}, the rational map is nonconstant.
By Lemmas \ref{L:rational maps from a product}
and~\ref{L:no rational maps}
it is the $i$th projection for some $i$.
The corresponding point in $X(K)$ is $(u_i,v_i)$.
\item
Suppose that $\{i,j\} \in E$.
The same argument as for~\eqref{I:points of X}
shows that $Y_{u_i u_j}(\FF(G)) = Y_{u_i u_j}(K) \union \{(x_{ij},y_{ij})\}$,
the last point coming from the identity $Z_{\{i,j\}} \to Y_{u_i u_j}$.
By~\eqref{I:X has real points}, 
we may embed $K$ in $\R$ so that the $u_i$ are mapped
to algebraically independent real numbers so close to zero
that $Y_{u_i u_j}(\R) = \emptyset$ by~\eqref{I:Y has no real points}.
Thus $Y_{u_i u_j}(K)=\emptyset$.
So $Y_{u_i u_j}(\FF(G)) = \{(x_{ij},y_{ij})\}$.

The argument for $Y_{u_i+u_j}(\FF(G))=\emptyset$
is the same, except now that 
$Z_{\{i,j\}}$ is not birational to $Y_{u_i u_j}$.
\item 
The argument is the same as for~\eqref{I:points of Y for edges in E}.\qedhere
\end{enumerate}
\end{proof}

%****************************************************************************
\section{Construction of the inverse functor}\label{S:inverse functor}

\subsection{Construction}

Let $\E$ be the essential image of $\FF$.
We may view $\FF$ as a functor $\Graphs \to \E$.
We now construct an essential inverse $\GG \colon \E \to \Graphs$ of $\FF$.

Suppose that $F$ is an object of $\E$.
Define $V \colonequals X(F)$.
For $i \in V = X(F)$, let $u_i$ be the first coordinate of $i$.
For $\{i,j\} \in \binom{V}{2}$,
Lemma~\ref{L:points over F(G)}(\ref{I:points of Y for edges in E},\ref{I:points of Y for edges outside E})
shows that exactly one of $Y_{u_i u_j}$ and $Y_{u_i+u_j}$ has an $F$-point.
Let $E$ be the set of $\{i,j\} \in \binom{V}{2}$
for which it is $Y_{u_i u_j}$ that has a $F$-point.
Let $\GG(F)$ be the graph $(V,E)$.

Suppose that $f \colon F \to F'$ is a morphism of $\E$.
We want to define a morphism of graphs $g \colon \GG(F) \to \GG(F')$. 
On vertices, $g$ is the map $X(F) \to X(F')$ induced by $f$.
If $\{i,j\}$ is an edge of $\GG(F)$,
then $Y_{u_i u_j}$ has an $F$-point,
and applying $f$ shows that $Y_{f(u_i) f(u_j)}$ has an $F'$-point,
so $\{g(i),g(j)\}$ is an edge of $\GG(F')$.
If $\{i,j\}$ is \emph{not} an edge of $\GG(F)$,
then the analogous argument using $Y_{u_i+u_j}$
shows that $\{g(i),g(j)\}$ is \emph{not} an edge of $\GG(F')$.
Thus $g$ is a morphism of graphs.
Define $\GG(f) \colonequals g$.
This completes the specification of a functor~$\GG$.

\subsection{Properties}

For each graph $G$, let $\epsilon_G \colon G \to X(\FF(G)) = \GG(\FF(G))$ 
be the map sending $i$ to $(u_i,v_i)$.

\begin{proposition}
\label{P:GF=1}
We have $\GG \FF \isom 1_{\Graphs}$.
\end{proposition}

\begin{proof}
By Lemma~\ref{L:points over F(G)}, $\epsilon_G$ is a bijection,
and in fact a graph isomorphism since the following are equivalent
for a pair $\{i,j\}$:
\begin{itemize}
\item $\{i,j\}$ is an edge of $G$;
\item $Y_{u_i u_j}(\FF(G)) \ne \emptyset$;
\item There is an edge between the vertices
$(u_i,v_i)$ and $(u_j,v_j)$ of $\GG(\FF(G))$.
\end{itemize}
The isomorphism $\epsilon_G$ varies functorially with $G$.
\end{proof}

\begin{proposition}
\label{P:G is faithful}
The functor $\GG$ is faithful.
\end{proposition}

\begin{proof}
We must show how to recover a morphism $f \colon F \to F'$ in $\E$
given $F$, $F'$, and $\GG(f)$.
Without loss of generality, replace $F$ by an isomorphic field
to assume that $F = \FF(G)$, which is generated by elements
$u_i$, $v_i$, $x_{ij}$, $y_{ij}$.
Since $\GG(f)$ is the map $X(F) \to X(F')$ induced by $f$,
we recover $f(u_i)$ and $f(v_i)$ for all $i$.
For each edge $\{i,j\}$ of $G$, 
the homomorphism $f$ maps $(x_{ij},y_{ij}) \in Y_{u_i u_j}(F)$
to a point of $Y_{u_i u_j}(F')$,
but by Lemma~\ref{L:points over F(G)}\eqref{I:points of Y for edges in E}
that point is unique, so we recover $f(x_{ij})$ and $f(y_{ij})$.
Similarly, for each non-edge $\{i,j\}$ of $G$,
Lemma~\ref{L:points over F(G)}\eqref{I:points of Y for edges outside E}
lets us recover $f(x_{ij})$ and $f(y_{ij})$.
Together, these determine $f$.
\end{proof}

\begin{proof}[Proof of Theorem~\ref{T:fully faithful}]
Propositions \ref{P:GF=1} and~\ref{P:G is faithful}
formally imply that $\FF$ is fully faithful.
\end{proof}

%****************************************************************************
\section{Computability}
\label{S:computability}

The specification of $\FF$ in Section~\ref{S:construction of F} 
sufficed for Theorem~\ref{T:fully faithful}.
But for the applications to computable model theory, 
for $G \in \Graphsw$
we need to modify the definition of $\FF(G)$ 
to ensure in particular that its domain is $\omega$ 
and not some other countable set.
To do this, 
we will iteratively build a standard presentation of a function field
starting from a presentation of its constant field.

Let $\pi \colon \omega \times \omega \isomto \omega$
be the standard pairing function.
We will repeatedly use the following: Given a field $k$
with $\dom(k) \subseteq \omega$,
and given an irreducible polynomial $f(x,y) \in k[x,y]$ 
of $y$-degree $d \in \Z_{>0}$,
the function field of the integral curve $f(x,y)=0$ is
\[
	k(x)[y]/(f(x,y)) 
	\isom k(x) \directsum k(x) y \directsum \cdots \directsum k(x) y^{d-1},
\]
and the presentation of $k$ can be extended to a presentation of 
this function field on the disjoint union $\dom(k) \sqcup \omega$.

Let $G=(\omega,E) \in \Graphsw$.
Partition the $\omega$ that is to be $\dom(\FFw(G))$
into countably many infinite columns 
$\omega^{[n]} \colonequals \pi(\{n\} \times \omega)$.
Define field operations on $K_0 \colonequals \omega^{[0]}$
so that $K_0$ is computably isomorphic to $\Q$.
For each $i\in\omega$,
use the elements of $\omega^{[2i+2]}$ to extend the field $K_i$
to $K_{i+1} \colonequals K_i(X) = K_i(u_{i+1},v_{i+1})$
in a uniform way such that $\pi(u_{i+1},v_{i+1}) > \pi(u_i,v_i)$.
Then $K \isom \Union K_i$, 
which has domain $\Union_{n \textup{ even}} \omega^{[n]}$.
Let $F_0 \colonequals K$.
Next, order the pairs $(i,j)$ with $i>j$ in lexicographic order, 
hence in order type $\omega$.  If $(i,j)$ is the $k$th such pair 
(so $k=\frac{i(i-1)}2+j+1$), 
inductively define $F_k$ by adjoining the elements of $\omega^{[2k-1]}$ 
to $F_{k-1}$ in a uniform way 
to form the function field $F_{k-1}(Z_{\{i,j\}})$.
Then define $\FF_\omega(G) \colonequals \Union F_k$, 
which has domain $\Union \omega^{[n]} = \omega$.
The action of $\FF_\omega$ on morphisms is defined as for $\FF$.

Really all we have done is to identify $\dom(\FF(G))$ with $\omega$.
In fact, we might as well modify the definition of $\FF$
so that $\dom(\FF(G)) = \omega$.
Then $\FF$ restricts to $\FFw$.

We record a property of the construction 
that will be useful in the proof of Theorem~\ref{thm:relations}:

\begin{lemma}
\label{L:mu}
Fix $G \in \Graphsw$.
Let $\mu \colon G \to \FFw(G)$ 
be the injection sending $i$ to $u_i$.
Then $\mu$ is computable, 
its range $\mu(G)$ is a computable subset of $\FFw(G) = \omega$,
and $\mu^{-1}$ (defined on $\mu(G)$) is computable.
Moreover, $\mu$ varies functorially with $G$.
\end{lemma}

\begin{proof}
The uniformity in the construction of the $K_i$ in $\FFw(G)$
ensures that $\mu$ is computable even when $G$ is not.
(In fact, $\mu$ is independent of $G$.)
Given $j \in \FFw(G)$, we can determine if
$j \in \omega^{[2i+2]}$ for some $i \in \omega$;
if so, then compute $u_i$, and check if $j=u_i$.
This lets us check if $j \in \mu(G)$,
and if so computes $i$ such that $\mu(i)=j$.
\end{proof}

For $F \in \Fieldsw$, the injection 
\[
	X(F) \subseteq F \times F = \omega \times \omega 
        \stackrel{\pi}\To \omega
\]
defines a well-ordering on the set $X(F)$.

\begin{lemma}
\label{L:enumeration of X(F)}
If $F \in \Ew$, then there is an order-preserving bijection
$\delta_F \colon \omega \to X(F)$,
and $\delta_F$ and $\delta_F^{-1}$ are computable uniformly from an $F$-oracle.
\end{lemma}

\begin{proof}
We have $F \isom \FFw(G)$ for some $G \in \Graphsw$.
By Lemma~\ref{L:points over F(G)}\eqref{I:points of X},
$\# X(F) = \# X(\FFw(G)) = \#G = \aleph_0$.
Given $F$, the elements of $X(F)$ can be enumerated by searching in order;
call the $i$th element $\delta_F(i)$ (starting with $i=0$).
This defines $\delta_F$ and shows that
$\delta_F$ and $\delta_F^{-1}$ are computable.
\end{proof}

\begin{lemma}
\label{L:order-preserving}
If $G \in \Graphsw$,
then the bijection $\epsilon_G \colon G \to X(\FFw(G))$ 
is order-preserving, and $\epsilon_G = \delta_{\FFw(G)}$.
\end{lemma}

\begin{proof}
The bijection $\epsilon_G$ 
is order-preserving by the condition $\pi(u_{i+1},v_{i+1}) > \pi(u_i,v_i)$.
Since $\epsilon_G$ and $\delta_{\FFw(G)}$
are both order-preserving bijections $\omega \to X(\FFw(G))$,
they are equal.
\end{proof}

Let us now define the promised ``inverse'' functor 
$\GGw \colon \Ew \to \Graphsw$.
If $F$ is an object of $\Ew$,
then by transport of structure 
across the bijection $\delta_F \colon \omega \to X(F)$,
the graph $\GG(F)$ with vertex set $X(F)$
becomes a graph $\GGw(F)$ with vertex set $\omega$.
If $f \colon F \to F'$ is a morphism of $\Ew$,
then again by transport of structure,
the morphism $\GG(f) \colon \GG(F) \to \GG(F')$
becomes $\GGw(f) \colon \GGw(F) \to \GGw(F')$.

\begin{proposition}
\label{P:GwFw=1}
We have $\GGw \FFw = 1_{\Graphsw}$.  
\end{proposition}

\begin{proof}
If $G$ is an object of $\Graphsw$, 
the composition of sets
\[
	G \stackrel{\epsilon_G}\To X(\FFw(G)) \stackrel{\delta_{\FFw(G)}^{-1}}\To \GGw(\FFw(G))
\]
is the identity $\omega \to \omega$ by Lemma~\ref{L:order-preserving}.
Each step is functorial in $G$ by construction.
\end{proof}

\begin{proposition}
\label{P:Gw is computable}
The functor $\GGw \colon \Ew \to \Graphs$ is computable.
\end{proposition}

\begin{proof}
Given $F \in \Ew$, the construction of $\GGw(F)$ is effective,
as we now explain.
To compute whether a given pair $\{i,j\}$ is an edge of $\GGw(F)$,
first use Lemma~\ref{L:enumeration of X(F)} 
to compute $(u_i,v_i) \colonequals \delta_F(i)$
and $(u_j,v_j) \colonequals \delta_F(j)$.
By Lemma~\ref{L:points over F(G)}(\ref{I:points of Y for edges in E},\ref{I:points of Y for edges outside E}),
exactly one of $Y_{u_i u_j}$ and $Y_{u_i+u_j}$ has an $F$-point.
To find out which, search both curves in parallel.
When a point is found, which curve it is on tells us
whether $\{i,j\}$ is an edge.

Given fields $F,F'$ and a morphism $f \colon F \to F'$ in $\Ew$,
the morphism $\GGw(f)$ is the composition
\[
	\omega \stackrel{\delta_F}\To X(F) \stackrel{f}\To X(F') \stackrel{\delta_{F'}^{-1}}\To \omega,
\]
which is computable by Lemma~\ref{L:enumeration of X(F)}.
\end{proof}

Next is an effective version of Proposition~\ref{P:G is faithful}.

\begin{proposition}
\label{P:recover f from G(f)}
There exists a Turing functional that 
given $F,F' \in \E$ and $g \colon \GGw(F) \to \GGw(F')$
computes the unique morphism $f \colon F \to F'$ of $\Ew$
such that $\GGw(f) = g$.
\end{proposition}

\begin{proof}
Existence and uniqueness of $f$ follow from
Theorem~\ref{T:fully faithful}.
Define $u_i,v_i \in F$ by $(u_i,v_i)=\delta_F(i) \in X(F)$.
Since $\GGw(f)=g$, the map $X(F) \to X(F')$ induced by $f$ is 
the composition
\[
	X(F) \stackrel{\delta_F^{-1}}\To \GGw(F) 
	\stackrel{g}\To \GGw(F') \stackrel{\delta_{F'}}\To X(F').
\]
Thus we may compute $f(u_i)$ and $f(v_i)$ for any given $i \in \omega$.
For each $\{i,j\}$, let $(x_{ij},y_{ij})$ 
be the point of $Y_{u_i u_j}(F)$ or $Y_{u_i+u_j}(F)$,
according to whether $\{i,j\}$ is an edge of $\GGw(F)$ or not.
Then $f$ maps $(x_{ij},y_{ij})$ to the unique point of
$Y_{f(u_i) f(u_j)}(F')$ or $Y_{f(u_i)+f(u_j)}(F')$,
and that point can be found by a search,
so we can compute $f(x_{ij})$ and $f(y_{ij})$.
Finally, given any $z \in F$,
search for $u$'s, $v$'s, $x$'s, $y$'s as above and a rational function
expressing $z$ in terms of them;
evaluate the same rational function on their images under $f$
to compute $f(z)$ in $F'$.
\end{proof}

\begin{proposition}
\label{P:Fw is computable}
The functor $\FFw \colon \Graphsw \to \Fieldsw$ is computable.
\end{proposition}

\begin{proof}
The constructions of the fields $K_i$ and $F_i$ in Section~\ref{S:construction}
are done by a uniform process, 
so $\FFw(G)$ has been defined uniformly effectively below a $\Delta(G)$-oracle
(specifically, below the set $E$ of edges in $G$).
This provides the Turing functional $\Phi$ 
in Definition~\ref{defn:computablefunctor}) for $\FFw$.

Now suppose that we are given graphs $G$ and $G'$
(or rather, their atomic diagrams)
and a morphism $g \colon G \to G'$.
By the previous paragraph,
$\FFw(G)$ and $\FFw(G')$ are computable from these.
Also, $\GGw(\FFw(g))$ is computable, 
since by Proposition~\ref{P:GwFw=1} it equals $g$.
By Proposition~\ref{P:recover f from G(f)},
we can compute $\FFw(g)$.
\end{proof}

\begin{proposition}
\label{P:FwGw}
The composition $\FFw \GGw$ is computably isomorphic to $1_{\Ew}$.
\end{proposition}

\begin{proof}
For each $F \in \Ew$, Proposition~\ref{P:GwFw=1} yields
$\GGw(\FFw(\GGw(F))) = \GGw(F)$.
Applying Proposition~\ref{P:recover f from G(f)} to the equality in
each direction
yields an isomorphism $\FFw(\GGw(F)) \to F$ and its inverse,
and shows that they are computable from $F$.
Moreover, the construction is functorial in $F$.
\end{proof}

\begin{proof}[Proof of Theorem~\ref{T:computable functors}]
Combine Propositions \ref{P:Fw is computable}, \ref{P:Gw is computable}, 
\ref{P:GwFw=1}, and~\ref{P:FwGw}.
\end{proof}

\begin{proof}[Proof of Proposition~\ref{prop:extrafields}]
\hfill
\begin{enumerate}[\upshape (a)]
\item 
Theorem~\ref{T:computable functors}\eqref{I:Fw and Gw are inverse} 
shows that $\FFw \colon \Fieldsw \to \Graphsw$ is fully faithful;
from this, the result follows formally.
\item 
By Corollary~\ref{C:G and f are computable},
there exists $G \in \II$ with $G \le_T F$.

First suppose that $G$ is automorphically trivial.
Then so is every other $G' \in \II$.
For any automorphically trivial $G'$,
the presence of $\{i,j\}$ as an edge in $G'$
is determined by the answers to the questions
of the form ``Is $i=m$?'' or ``Is $j=m$?'' 
for $m$ in some finite set,
so $G'$ is computable.

Now suppose instead that $G$ is not automorphically trivial.
Trivially, $\Spec G$ contains $\deg G$.
A theorem of Knight, namely \cite{K86}*{Theorem~4.1},
states that if a structure is not
automorphically trivial, then its spectrum is upwards closed.
Thus $\Spec G$ contains $\deg F$ too.
In other words, there exists $G' \in \II$ with $\deg G' = \deg F$.\qedhere
\end{enumerate}
\end{proof}

%****************************************************************************
\section{Ordered fields}
\label{S:ordered fields}

\begin{proposition}
\label{P:F(G) in R}
For each $G \in \Graphsw$, 
the field $\FFw(G)$ is isomorphic
to a subfield of $\R$.
Moreover, the field homomorphism $\FFw(G) \injects \R$ may be chosen 
so that the induced ordering on $\FFw(G)$ is computable from $G$.
\end{proposition}

\begin{proof}
Given a subfield $k \subseteq \R$, an integral curve $Z$ over $k$,
and a point $z \in Z(\R)$ 
having at least one coordinate transcendental over $k$,
we obtain a $k$-embedding $k(Z) \injects \R$
by sending $f$ to $f(z)$.
We will use this observation inductively to show that all
the fields $K_i$ and $F_i$ in Section~\ref{S:construction} embed into $\R$.

First, $K_0 \colonequals \Q$ is a subfield of $\R$.
By the Lindemann--Weierstrass theorem, 
the numbers $e_i \colonequals \exp(2^{1/i})$ for $i \in \{2,3,\ldots\}$
are algebraically independent over $\Q$.
By Lemma~\ref{L:properties of the curves}\eqref{I:X has real points},
$X$ has a real point with $u$-coordinate $10 e_{2j+2}$;
choose the one with least $v$-coordinate.
This lets us inductively extend $K_{j-1} \injects \R$ 
to $K_j = K_{j-1}(X) \injects \R$ 
so that $u_j$ maps to $10 e_{2j+2}$.
Taking the union embeds $F_0 = K$ in $\R$.
For the $k$th pair $i>j$,
we have $u_i u_j, u_i+u_j \in [20,\infty)$ and $e_{2k+3} \in [1,2]$,
so by Lemma~\ref{L:properties of the curves}\eqref{I:20 to infinity},
the curve $Z_{\{i,j\}}$ has a real point with $x$-coordinate $e_{2k+3}$;
choose the one with least $y$-coordinate.
Thus for each $k \ge 1$ we may extend the embedding $F_{k-1} \injects \R$
to $F_k \injects \R$.
Taking the union yields an embedding $\FFw(G) \injects \R$.

Given two distinct elements of $\FFw(G)$, 
expressed in terms of the generators of $\FFw(G)$,
we may compute them numerically until we determine which is greater.
\end{proof}

Unfortunately, Proposition~\ref{P:F(G) in R}
implies neither that $\FFw$ is a functor to the category
of ordered fields (with order-preserving field homomorphisms), 
nor that our completeness results hold for ordered fields.
The problem is that isomorphisms in the category of ordered fields 
are much more restricted than isomorphisms in the category of fields.
In fact, we have the following.

\begin{proposition}
There is no faithful functor from $\Graphsw$ 
to the category of ordered fields,
or even to the category of ordered sets.
\end{proposition}

\begin{proof}
Every (order-preserving) automorphism of a finite totally ordered set is trivial.
Therefore every finite-order automorphism of a totally ordered set
acts trivially on every orbit, and hence is trivial.
On the other hand, 
there exist countable graphs with nontrivial finite-order automorphisms.
\end{proof}

%****************************************************************************
\section{Consequences in computable model theory}
\label{sec:consequences}

Many of the results below rely on \cite{HKSS}*{Theorem~3.1}, which shows that 
for every countable structure $\calA$ that is not automorphically trivial,
there exists a graph $G \in \Graphsw$ with the same spectrum 
and the same $\bfd$-computable dimension for each Turing degree $\bfd$,
and with the property that for every relation $R$ on $\calA$, 
there exists a relation on $G$ 
with the same degree spectrum as $R$ on $\calA$.

\subsection{Turing degree spectrum}
\label{S:Turing degree spectrum}

Recall from Section~\ref{S:notation computable model theory} 
the definition of the spectrum: 
\[
	\Spec M \colonequals \{\deg N: N \isom M \textup{ and } \dom(N) = \omega\}.
\]
\begin{proposition}
\label{P:Turing degrees}
\hfill
\begin{enumerate}[\upshape (a)]
\item \label{I:same Turing degree}
If $G \in \Graphsw$, then $G \equiv_T \FFw(G)$.
\item \label{I:same Turing degree for morphisms}
If $g \colon G \to G'$ is a morphism between computable graphs in $\Graphsw$,
then $g \equiv_T \FFw(g)$.
\end{enumerate}
\end{proposition}

\begin{proof}
Theorem~\ref{T:computable functors} showed that $\FFw$ and $\GGw$
are computable functors.
\end{proof}

\begin{theorem}
\label{T:spectra}
Let $G \in \Graphsw$.
\begin{enumerate}[\upshape (a)]
\item \label{I:G is automorphically trivial}
If $G$ is automorphically trivial, then $\Spec G = \{\boldsymbol{0}\}$, 
and $\Spec \FFw(G)$ contains all Turing degrees.
\item \label{I:G is not automorphically trivial}
If $G$ is not automorphically trivial, then $\Spec G = \Spec \FFw(G)$.
\end{enumerate}
\end{theorem}

\begin{proof}\hfill
\begin{enumerate}[\upshape (a)]
\item
Every automorphically trivial graph is computable,
so $\Spec G = \{\boldsymbol{0}\}$.
The field $\FFw(G)$ is computable but not automorphically trivial;
Knight's theorem \cite{K86}*{Theorem~4.1} applied to $\FFw(G)$ 
yields the result.
\item 
Applying Proposition~\ref{P:Turing degrees}\eqref{I:same Turing degree} 
to every $G'$ isomorphic to $G$
yields $\Spec G \subseteq \Spec \FFw(G)$.
On the other hand, since $G$ is not automorphically trivial, 
Proposition~\ref{prop:extrafields}\eqref{I:dichotomy}
yields $\Spec \FFw(G) \subseteq \Spec G$.\qedhere
\end{enumerate}
\end{proof}

\begin{corollary}
\label{C:spectrum}
For every countable structure $\calA$ that is not automorphically trivial,
there exists $F \in \Fieldsw$ with $\Spec F = \Spec \calA$.
\end{corollary}

\begin{proof}
Since $\calA$ is not automorphically trivial,
\cite{HKSS}*{Theorem~3.1} yields some $G \in \Graphsw$ 
such that $\Spec G = \Spec \calA\neq\{\bfz\}$.
By Knight's theorem, $G$ is not automorphically trivial,
so Theorem~\ref{T:spectra} shows $\Spec G = \Spec \FFw(G)$.
Let $F \colonequals \FFw(G)$.
\end{proof}

\subsection{Computable categoricity}
\label{S:computable categoricity}

Let $\calA$ be a computable structure, let $\bfd$ be a Turing degree,
and let $\alpha$ be a computable ordinal.
If $\calB$ is \emph{any} structure on $\omega$,
let $\calB^{(\alpha)}$ denote the $\alpha$-jump of the atomic diagram of $\calB$.
Let $\Isom_{\bfd}(\calA)$ be the set of $\bfd$-computable isomorphism classes
in the set of computable structures isomorphic to $\calA$;
then the cardinal $\#\Isom_{\bfd}(\calA) \in \{1,2,\ldots,\aleph_0\}$ is called 
the \defi{$\bfd$-computable dimension} of $\calA$ \cite{HKSS}*{Definition~1.2}.
The \defi{categoricity spectrum} of $\calA$
is the set of Turing degrees $\bfd$ 
such that the $\bfd$-computable dimension of $\calA$ is~$1$.
Finally, $\calA$ is \defi{relatively computably categorical}
if \emph{every} structure $\calB\cong\calA$ with domain $\omega$
is $\calB$-computably isomorphic to $\calA$,
and \defi{relatively $\Delta^0_\alpha$-categorical}
if every structure $\calB\cong\calA$ with domain $\omega$
is $\calB^{(\alpha)}$-computably isomorphic to $\calA$.

\begin{theorem}
\label{thm:effdim}
For every computable structure $\calA$,
there exists a computable field $F$ such that
\begin{enumerate}[\upshape (i)]
\item for each Turing degree $\bfd$, 
the field $F$ has the same $\bfd$-computable dimension as $\calA$;
\item $F$ has the same categoricity spectrum as $\calA$; and
\item for every computable ordinal $\alpha$, 
the field $F$ is relatively $\Delta_\alpha^0$-categorical if and only if $\calA$ is.
\end{enumerate}
\end{theorem}

\begin{proof}
First suppose that $\calA$ is automorphically trivial.
Let $F=\Q$.
Then $\calA$ has $\bfd$-computable dimension~$1$ for every $\bfd$, 
and $\calA$ is relatively $\Delta_\alpha^0$-categorical
for all $\alpha$,
and the field $\Q$ has the same properties.

{}From now on, suppose that $\calA$ is not automorphically trivial.
Use \cite{HKSS}*{Theorem~3.1} to replace $\calA$
by a computable graph $G$ on domain $\omega$,
and let $F = \FFw(G)$.

(i) The functors $\FFw$ and $\GGw$ are computable 
(Theorem~\ref{T:computable functors}),
so they map computable objects of $\Graphsw$
to computable objects of $\Ew$ and vice versa,
and they respect isomorphism and $\bfd$-computable isomorphism
of such objects
(it will be OK to work in $\Ew$ instead of $\Fieldsw$ 
since all fields in $\Fieldsw$ isomorphic to $F$ are in $\Ew$).
Thus they induce maps between the sets $\Isom_{\bfd}(G)$ and $\Isom_{\bfd}(F)$.
The composition of these maps in either order is the identity
since $\GGw(\FFw(G'))$ is computably isomorphic (in fact, equal) to $G'$ 
for every $G' \in \Graphsw$ 
and $\FFw(\GGw(F'))$ is computably isomorphic to $F'$ for every $F' \in \Ew$
(Theorem~\ref{T:computable functors}\eqref{I:Fw and Gw are inverse}).
Thus $\#\Isom_{\bfd}(G)=\#\Isom_{\bfd}(F)$.

(ii) This follows from (i), since the categoricity spectrum is 
defined as the set of $\bfd$ 
such that the $\bfd$-computable dimension equals $1$.

(iii)
Fix $\alpha$.
If $G$ is not relatively $\Delta^0_\alpha$-categorical,
then some graph $G'$ violates the condition in the definition above,
and $\FFw(G)$ and $\FFw(G')$ violate the same condition, 
so $\FFw(G)$ is not relatively $\Delta^0_\alpha$-categorical either.
On the other hand, if it holds of $G$, then it holds immediately of $\FFw(G)$
and every other field $\FFw(G')$ with $G'\isom G$.  
But for a field $F\isom\FFw(G)$,
Corollary~\ref{C:G and f are computable} shows that there exists 
an $F$-computable isomorphism $f \colon F \to \FFw(H)$
for some graph $H$ that is isomorphic to $G$ and Turing-computable from $F$.
Therefore, there is a $(\deg H)^{(\alpha)}$-computable 
(hence $(\deg F)^{(\alpha)}$-computable)
isomorphism $g$ from $H$ onto $G$, 
whose image $\FFw(g)$ is a $(\deg F)^{(\alpha)}$-computable
isomorphism from $\FFw(H)$ onto $\FFw(G)$.  So the composition $\FFw(g)\circ f$
is a $(\deg F)^{(\alpha)}$-computable isomorphism from $F$ onto $\FFw(G)$,
proving that $\FFw(G)$ is also relatively $\Delta^0_\alpha$-categorical.
\end{proof}

\begin{corollary}
Fields realize all computable dimensions $\le \omega$.
\end{corollary}

\begin{proof}
Theorems of Goncharov \cites{G77,G80} 
and others have shown long since that
computable structures (in fact, graphs) can have every computable
dimension from $1$ up to $\aleph_0$.  (See \cite{AK00}*{\S 12.6}
for a summary of these results using the terminology of this article.)
Theorem~\ref{thm:effdim} allows us to carry this over to fields.
\end{proof}

The basic definition of computable categoricity of a computable structure $\calA$
shows computable categoricity to be a $\Pi^1_1$ property:
it is defined by a statement using a universal quantifier over \emph{sets} of natural
numbers.  One can view it as saying that, for all subsets $f$ of $\omega^2$ and all $e\in\omega$,
either the $e$th computable function $\varphi_e$ fails to define a computable structure,
or the function (if any) defined on $\omega$ by $f$ fails to be an isomorphism from $\calA$
onto the structure defined by $\varphi_e$, or else there exists a Turing program which computes
an isomorphism from that structure onto $\calA$.  All quantifiers over natural numbers
can be absorbed into the universal quantifier over sets, so we view this formula
as universal over sets, i.e., $\Pi^1_1$, with the superscript $1$ signifying quantification
over sets.

Sometimes a $\Pi^1_1$ property is expressible also by a simpler formula.  
In \cite{DKLLMT}*{Theorem~1}, however, 
Downey, Kach, Lempp, Lewis, Montalb\'an, and Turetsky
give a proof that computable categoricity for trees is $\Pi^1_1$-complete.  
Theorem~\ref{T:computable functors} allows us to transfer this result to fields.

\begin{theorem}
\label{thm:Pi11}
The property of computable categoricity for computable fields
is $\Pi^1_1$-complete.  
\end{theorem}
That is, 
\[
	\set{e\in\omega}{\varphi_e\text{~computes the atomic diagram of a
computably categorical field}}
\]
is a $\Pi^1_1$ set, and every $\Pi^1_1$ set is $1$-reducible to this set.

In contrast, the property of computable catgeoricity for algebraic fields
is just $\Pi^0_4$-complete (see \cite{HKMS}*{Theorem~5.4}).  
For fields of infinite transcendence degree, it was shown only recently,
in \cite{MS}*{Theorem~3.4}, 
that they could be computably categorical at all, and it was not known whether
they could be computably categorical without being relatively computably
categorical, which is a $\Sigma^0_3$ property.  
So Theorem~\ref{thm:Pi11} represents a significant step forward.

\begin{proof}
Theorem~3.1 of \cite{HKSS} enables us to build a computable graph
corresponding to an arbitrary computable tree, effectively in the
index of the tree, and then Theorem~\ref{T:computable functors} 
builds a computable
field that is computably categorical if and only if the graph
(and hence the original tree) was.  Thus we have a $1$-reduction
from a $\Pi^1_1$-complete set (the indices of computably categorical trees)
to the set of indices of computably categorical fields.
\end{proof}

\subsection{Degree spectrum}

The \defi{degree spectrum} $\DS{\calA}{R}$ of a relation $R$ on a computable
structure $\calA$ is the set of all Turing degrees of images of $R$
under isomorphisms $f$ from $\calA$ onto computable structures $\calB$.
(Understand that here the relation $R$ is generally not in the language;
if it were, then this image would always be computable, just by the
definition of a computable structure.  
The relation $R$ may be \emph{definable} in
the language of the structure, however, in which case its definition places an
upper bound on the complexity of these images.)  
If $f$ is computable, then $f(R)$ is Turing-equivalent to $R$ itself;
thus it is the noncomputable $f$ that give this definition its teeth.
Here we show that all degree spectra of relations on 
automorphically nontrivial computable structures 
can be realized as degree spectra of relations on fields:

\begin{theorem}
\label{thm:relations}
Let $\calA$ be any computable structure that is not automorphically
trivial, and let $R$ be an $n$-ary relation on $\calA$.  
Then there exists a field $F$ and an $n$-ary relation $S$ on $F$ such that
\[
	 \DS{\calA}{R} = \DS{F}{S}.
\]
\end{theorem}

\begin{proof}
As usual, we appeal to \cite{HKSS}*{Theorem~3.1} to assume
that $\calA$ is in fact a graph, hereafter named $G$.
For this proof, Theorem~\ref{T:computable functors} is not quite enough;
to transfer relations we need also Lemma~\ref{L:mu},
about the map $\mu \colon G \to \FFw(G)$ and its inverse.
Let $F = \FFw(G)$, 
and let $S = \mu(R) \subseteq F^n$ (apply $\mu$ coordinate-wise).

Suppose that $g \colon G \to G'$ is an isomorphism of
computable graphs in $\Graphsw$.
In the diagram of sets
\[
\xymatrix{
G \ar[d]_g \ar[r]^-\mu & \FFw(G) \ar[d]^{\FFw(g)} \\
G' \ar[r]^-\mu & \FFw(G'), \\
}
\]
$R$ and $S$ map downwards to relations $R'$ on $G'$
and $S'$ on $\FFw(G')$,
and $\DS{G}{R}$ consists of the degrees $\deg R'$
arising in this way from all possible $g$.
By functoriality of $\mu$, the diagram commutes,
so the bottom horizontal $\mu$ maps $R'$ to $S'$.
Since $\mu$ and $\mu^{-1}$ are computable, $R' \equiv_T S'$.
Thus $\deg R' = \deg S' \in \DS{F}{S}$.
Hence $\DS{G}{R}\subseteq\DS{F}{S}$.

Now suppose that $f \colon F \to F'$ is an isomorphism of computable fields
in $\Fieldsw$;
it maps $S$ to some $S'$;
then $\DS{F}{S}$ consists of the degrees $\deg S'$
arising in this way.
By Corollary~\ref{cor:extrafields},
there is a computable graph $G' \isom G$ and a computable
isomorphism $i \colon F' \to \FFw(G')$.
Composing $f$ with a computable isomorphism does not change 
the resulting Turing degree $\deg S'$,
so we may assume that $F'$ \emph{equals} $\FFw(G')$.
Since $\FFw$ is fully faithful,
$f$ is $\FFw(g)$ for some isomorphism $g \colon G \to G'$.
This $g$ maps $R$ to some $R'$,
and the previous paragraph shows that $\deg S' = \deg R' \in \DS{G}{R}$.
Hence $\DS{F}{S}\subseteq\DS{G}{R}$.
\end{proof}

\begin{remark}
It is not true that 
for every automorphically nontrivial computable structure $\calA$,
there exists a computable field $F$ of characteristic~$0$
such that for every relation $S$ on $F$,
there exists a relation $R$ on $\calA$
with the same degree spectrum.
For example, suppose that $\calA$ 
is the random graph or the countable dense linear order;
then the degree spectrum of every relation on $\calA$ 
either is $\{\bfz\}$ or is upwards-closed under $\leq_T$;
see \cite{HM}*{Corollary~2.11 and Proposition~3.6}.
On the other hand, 
in every computable field $F$ of characteristic $0$, 
one can effectively locate each positive integer $n$ 
(meaning the sum $1+\cdots+1$ 
of the multiplicative identity with itself $n$ times --- not the
element $n$ of the domain $\omega$), and therefore the unary relation $S$
consisting of those $n$ that lie in the Halting Problem 
will have degree spectrum exactly $\{\bfz'\}$.
\end{remark}

The restriction to automorphically nontrivial structures $\calA$
in Theorem~\ref{thm:relations} 
can be bypassed for unary relations, since on an automorphically
trivial computable structure $\calA$, each unary relation has degree spectrum
either $\{\boldsymbol{0}\}$ (if the relation is either finite or cofinite)
or else the set of all Turing degrees.  Both of these can
easily be realized as spectra of unary relations on fields.
We leave the analysis of $n$-ary relations on such structures
for another time.

Functoriality also allows one to show
that a relation $R$ on a countable graph $G$ is
relatively intrinsically $\Sigma^0_\alpha$ if and only if
its image $h(R)$ (defined exactly as in the proof of 
Theorem~\ref{thm:relations}) is relatively intrinsically $\Sigma^0_\alpha$
on the field $\FFw(G)$; and similarly for relations
that are relatively intrinsically $\Pi^0_\alpha$,
$\Sigma^1_m$, etc.  Likewise, every definable
relation on $G$ is mapped to a definable relation
on $\FFw(G)$:  this is immediate if one views
our construction as a bi-interpretation between the graph
and the field.

\subsection{Automorphism spectrum}

In \cite{HMM}*{Definition~1.1}, Harizanov, Miller, and Morozov defined the
\defi{automorphism spectrum} of a computable structure
to be the set of all Turing degrees of nontrivial automorphisms
of the structure.  
They used \cite{HKSS}*{Theorem~3.1} to show that 
every automorphism spectrum is the automorphism spectrum of a computable graph.
This sets up another application of Theorem~\ref{T:computable functors}.

\begin{theorem}
\label{thm:autspec}
For every computable structure $\calA$, there is a computable field $F$
with the same automorphism spectrum as $\calA$.
\end{theorem}
\begin{proof}
The theorem follows from the full faithfulness of $\FFw$,
along with its preservation of Turing degrees of automorphisms
(Proposition~\ref{P:Turing degrees}\eqref{I:same Turing degree for morphisms}).
\end{proof}

%****************************************************************************
\appendix
\section{Algebraic geometry facts}

\begin{lemma}
\label{L:points over function field}
If $V$ and $W$ are varieties over a field $k$,
and $W$ is integral, 
then $V(k(W))$ is in bijection with the set
of rational maps $W \dashrightarrow V$.
\end{lemma}

\begin{proof}
The description of a point in $V(k(W))$ involves only finitely many
elements of $k(W)$, and there is a dense open subvariety $U \subseteq W$
on which they are all regular.
\end{proof}

\begin{lemma}
\label{L:two curves}
Let $k$ be a field of characteristic~$0$.
Let $C$ and $D$ be geometrically integral curves over $k$
such that $g_C=g_D > 1$.
Every nonconstant rational map $C \dashrightarrow D$ is birational.
\end{lemma}

\begin{proof}
This is a well known consequence of Hurwitz's formula.
\end{proof}

\begin{lemma}
\label{L:rational maps from a product}
Let $V_1,\ldots,V_n$ be geometrically integral varieties 
over a field $k$ of characteristic~$0$.
Let $C$ be a geometrically integral curve over $k$ such that $g_C>1$.
Then each rational map $V_1 \times \cdots \times V_n \dashrightarrow C$
factors through the projection
$V_1 \times \cdots \times V_n \to V_i$ for at least one $i$.
\end{lemma}

\begin{proof}
By induction, we may assume that $n=2$.
We may assume that $k$ is algebraically closed.
A rational map $\phi \colon V_1 \times V_2 \dashrightarrow C$
may be viewed as an algebraic family of rational maps $V_1 \dashrightarrow C$
parametrized by (an open subvariety of) $V_2$.
By the de Franchis--Severi theorem \cite{Lang1960}*{pp.~29--30}
there are only finitely many nonconstant rational maps 
$V_1 \dashrightarrow C$, so they do not vary in algebraic families.
Thus either the rational maps in the family are all the same,
in which case $\phi$ factors through the first projection,
or each rational map in the family is constant,
in which case $\phi$ factors through the second projection.
\end{proof}

\begin{remark}
Lemma~\ref{L:rational maps from a product} holds even if $\Char k=p$.
This can be deduced from the characteristic~$p$ analogue of
the de Franchis--Severi theorem \cite{Samuel1966}*{Th\'eor\`eme~2}:
even though the set of nonconstant rational maps $V_1 \dashrightarrow C$
can now be infinite (because of Frobenius morphisms),
they still do not vary in algebraic families.
\end{remark}

%****************************************************************************
% \section*{Acknowledgements} 

\end{document}